\documentclass[11pt]{article}
\usepackage[a4paper,margin=1in]{geometry}
\usepackage{amsmath,amssymb,amsthm}
\usepackage{hyperref}

\numberwithin{equation}{section}

\newtheorem{theorem}{Theorem}
\newtheorem{lemma}[theorem]{Lemma}
\newtheorem{proposition}[theorem]{Proposition}

\theoremstyle{remark}
\newtheorem{remark}[theorem]{Remark}

\title{\textbf{Optimizing the Upper-Bound Constant for the Crossing Number of Polynomial Curve Systems}}
\author{Hyungryul Baik}
\date{\today}

\begin{document}
\maketitle

\begin{abstract}
Baader, J\"org, and Parlier recently established an upper bound for the crossing number
of curve systems of size $m\asymp g^{1+\alpha}$ on a genus $g$ surface, obtaining a leading
coefficient of $9/4=2.25$. Their construction relies on fibre surfaces associated with
complete bipartite graphs and uses a symmetric parameter choice corresponding to the
central binomial coefficient. In this note, we optimize their construction by relaxing
the parameter symmetry and solving the resulting entropy balance problem.
We show that for every $\alpha>0$ and every $\varepsilon>0$,
\[
\mathrm{Cr}\bigl(g,\lfloor g^{1+\alpha}\rfloor\bigr)
\ \le\ (C_\star+\varepsilon)\,\alpha^2\, g^{1+2\alpha}(\log g)^2
\qquad (g\ \text{sufficiently large}),
\]
where
\[
C_\star\ =\ \inf_{0<x\le 1/2}\ \frac{2x}{H(x)^2}
\ \approx\ 1.5805443269,
\qquad
H(x)=-x\log x-(1-x)\log(1-x).
\]
This reduces the previous constant by about $30\%$ while staying within the same
topological framework.
\end{abstract}


\section{Introduction}
Throughout, $\log$ denotes the natural logarithm.
The problem of minimizing intersections among families of curves on surfaces is a natural
topological analogue of the classical graph crossing number problem, which dates back to
Tur\'an's brick factory problem and Zarankiewicz's work on complete bipartite graphs \cite{Zar}.
The planar crossing lemma (see e.g.\ \cite{Sze} and references therein) provides powerful lower
bounds for graphs, while curve systems on high-genus surfaces involve a rich interplay between
combinatorial topology and geometry.

For a closed orientable surface $\Sigma_g$ of genus $g$, consider a family
$\Gamma=\{\gamma_1,\dots,\gamma_m\}$ of $m$ pairwise non-isotopic simple closed curves.
The \emph{crossing number} of $\Gamma$ is
\[
\mathrm{cr}(\Gamma)\ :=\ \sum_{1\le i<j\le m} i(\gamma_i,\gamma_j),
\]
where $i(\cdot,\cdot)$ denotes geometric intersection number. The corresponding extremal function is
\[
\mathrm{Cr}(g,m)\ :=\ \min\bigl\{ \mathrm{cr}(\Gamma)\ \bigm|\ \Gamma
\ \text{as above with}\ \#\Gamma=m \bigr\}.
\]
Baader, J\"org, and Parlier \cite{BJP} determined $\mathrm{Cr}(g,m)$ for polynomial-size curve systems
$m\asymp g^{1+\alpha}$ up to absolute constants. Combining their constructive upper bound with the crossing
inequality of Hubard and Parlier \cite{HP}, they show that for $m=\lfloor g^{1+\alpha}\rfloor$,
\begin{equation}\label{eq:BJP-main}
\frac1{257}\,\alpha^2\, g^{1+2\alpha}(\log g)^2
\ \le\ \mathrm{Cr}\bigl(g,\lfloor g^{1+\alpha}\rfloor\bigr)
\ \le\ \frac94\,\alpha^2\, g^{1+2\alpha}(\log g)^2
\end{equation}
for all $\alpha\ge 0$ and all sufficiently large $g$.
The Baader--J\"org--Parlier upper bound is achieved by an explicit curve system on a fibre surface $\Sigma(p,q)$:
a ribbon neighbourhood of the complete bipartite graph $K_{p,q}$, using a symmetric choice that corresponds
to a central binomial coefficient. The purpose of this note is to show that the constant $\frac94$ is not
a limitation of the underlying fibre-surface framework, but rather a byproduct of that symmetric choice.
\medskip

More precisely, we identify the optimization problem that is implicit in the construction of \cite{BJP}.
Within the family of curve systems obtained as boundary curves of subsurfaces $\Sigma(2,k)\subset\Sigma(p,q)$
(with $q\asymp \log g$ and $k/q\to x\in(0,1/2]$), the number of curves is governed by the exponential growth of
$\binom{q}{k}$, while the crossing estimate depends linearly on $k$.
This yields an entropy--intersection tradeoff with asymptotic leading constant
\[
f(x)\ =\ \frac{2x}{H(x)^2},\qquad H(x)=-x\log x-(1-x)\log(1-x),
\]
and optimizing over $x$ produces
\[
C_\star=\inf_{0<x\le 1/2} f(x).
\]
In this sense, $C_\star$ is the minimum produced by this entropy-balance optimization of the fibre-surface
construction of \cite{BJP} (without claiming optimality among all possible constructions on $\Sigma_g$).

\begin{theorem}[Improved upper-bound constant]\label{thm:main}
Fix $\alpha>0$ and $\varepsilon>0$. Let
\[
H(x)\ :=\ -x\log x-(1-x)\log(1-x),\qquad x\in(0,1),
\]
and define
\[
C_\star\ :=\ \inf_{0<x\le 1/2}\ \frac{2x}{H(x)^2}.
\]
Then there exists $N\in\mathbb{N}$ such that for all $g\ge N$,
\[
\mathrm{Cr}\bigl(g,\lfloor g^{1+\alpha}\rfloor\bigr)
\ \le\ (C_\star+\varepsilon)\,\alpha^2\, g^{1+2\alpha}(\log g)^2.
\]
Numerically, $C_\star=1.5805443269\ldots$.
\end{theorem}

\begin{remark}[Hierarchy of improvements]
The improvement over the constant $\frac94=2.25$ in \cite{BJP} can be viewed in two steps.
If one keeps the symmetric choice $k=q/2$ but optimizes the scale of $q$, one obtains the constant
$1/(\log 2)^2 \approx 2.08137$. Allowing $k/q\neq 1/2$ and optimizing the ratio yields the smaller
constant $C_\star$ (up to $\varepsilon$).
\end{remark}

\section{Fibre surfaces, curve families, and basic estimates}
We briefly recall the combinatorial model used in \cite{BJP}, based on the ribbon surface model of
Baader \cite{Baa14}. Let $K_{p,q}\subset \mathbb{R}^3$ be a complete bipartite graph whose $p$ and $q$
vertices lie on two skew lines $U$ (upper) and $L$ (lower), respectively. Let $\Sigma(p,q)$ be a ribbon
neighbourhood of $K_{p,q}$, i.e.\ the union of $pq$ ribbons thickening the edges of $K_{p,q}$, as in
\cite{Baa14,BJP}. Its Euler characteristic satisfies
\[
\chi(\Sigma(p,q))\ =\ \chi(K_{p,q})\ =\ p+q-pq\ =\ 1-(p-1)(q-1),
\]
and $\Sigma(p,q)$ has nonempty boundary (a torus link of type $T(p,q)$).

\begin{lemma}[Embedding criterion]\label{lem:embed}
Let $F$ be a compact connected orientable surface with nonempty boundary and Euler characteristic $\chi(F)$.
If $|\chi(F)|\le 2g-2$, then $F$ embeds as a subsurface of the closed orientable surface $\Sigma_g$.
\end{lemma}
\begin{proof}
Write $F$ as a surface of genus $h$ with $b\ge 1$ boundary components. Then $\chi(F)=2-2h-b$, so
$|\chi(F)|=2h+b-2\ge 2h-1$. Hence $h\le (|\chi(F)|+1)/2\le (2g-1)/2<g$, i.e.\ $h\le g-1$.
By the classification of surfaces, any compact orientable surface of genus at most $g-1$ with boundary embeds
in $\Sigma_g$ as a subsurface.
\end{proof}

\medskip
Fix integers $p\ge 2$, $q\ge 2$, and an odd integer $k$ with $1\le k\le q$.
Inside $\Sigma(p,q)$ consider subsurfaces obtained by choosing
two consecutive vertices on the upper line $U$ (there are $p-1$ choices) and
choosing $k$ vertices among the $q$ vertices on the lower line $L$.
The corresponding ribbon neighbourhood is naturally homeomorphic to $\Sigma(2,k)$.
Since $\gcd(2,k)=1$ (as $k$ is odd), its boundary is connected; denote this boundary curve by $\gamma$.
Let $\Gamma(p,q;k)$ be the collection of all such boundary curves. Then
\begin{equation}\label{eq:M-def}
M(p,q;k)\ :=\ \#\Gamma(p,q;k)\ =\ (p-1)\binom{q}{k}.
\end{equation}
As observed in \cite{BJP}, any two distinct curves in $\Gamma(p,q;k)$ are non-isotopic: given $\gamma\neq\delta$,
there exists a vertex of $U$ or $L$ used in defining $\gamma$ but not $\delta$, and an essential properly embedded
arc intersecting $\gamma$ but disjoint from $\delta$.

\begin{lemma}[Crossing estimate]\label{lem:crossing}
For all integers $p\ge 2$, $q\ge 2$, and odd $1\le k\le q$,
\[
\mathrm{cr}\bigl(\Gamma(p,q;k)\bigr)\ \le\ \frac{4k}{p-1}\,M(p,q;k)^2.
\]
\end{lemma}
\begin{proof}
Fix $\gamma\in\Gamma(p,q;k)$. As in \cite[\S3]{BJP}, any curve $\delta\in\Gamma(p,q;k)$ whose chosen pair of upper
vertices is disjoint from that of $\gamma$ can be isotoped to be disjoint from $\gamma$ (potential intersections at
lower vertices can be removed by a small perturbation). Thus $\delta$ can intersect $\gamma$ only if their upper pairs
overlap, i.e.\ share one or two upper vertices.
There are exactly $\binom{q}{k}$ curves sharing both upper vertices with $\gamma$, and at most $2\binom{q}{k}$ curves
sharing exactly one upper vertex with $\gamma$ (the two adjacent choices of consecutive upper pairs). Moreover, the local
contribution to intersection number near a shared upper vertex is at most $2k$ (the relevant $K_{2,k}$-subgraph has degree
$k$ at that upper vertex), hence two curves intersect at most $2k$ times if they share one upper vertex and at most $4k$ times
if they share two. This $2k/4k$ bound depends only on $k$ and is independent of $q$.
Therefore the total number of intersections between $\gamma$ and all other curves is at most
\[
\bigl(2\binom{q}{k}\bigr)\cdot (2k)\ +\ \binom{q}{k}\cdot (4k)
\ =\ 8k\binom{q}{k}.
\]
Summing over all $\gamma$ and dividing by $2$ for double counting yields
\[
\mathrm{cr}\bigl(\Gamma(p,q;k)\bigr)
\ \le\ \frac12\,M(p,q;k)\cdot 8k\binom{q}{k}
\ =\ 4k\binom{q}{k}\,M(p,q;k).
\]
Using $M(p,q;k)=(p-1)\binom{q}{k}$ gives the claim:
\[
\mathrm{cr}\bigl(\Gamma(p,q;k)\bigr)
\ \le\ 4k (p-1)\binom{q}{k}^2
\ =\ \frac{4k}{p-1}\,M(p,q;k)^2.
\]
\end{proof}

\begin{lemma}[Stirling--entropy lower bound]\label{lem:entropy}
Fix $x\in(0,1)$. There exist constants $c_x>0$ and $Q_x\in\mathbb{N}$ such that for all integers $q\ge Q_x$
and all integers $k$ with $|k-xq|\le 2$,
\[
\binom{q}{k}\ \ge\ \frac{c_x}{\sqrt{q}}\;\exp\!\bigl(q\,H(x)\bigr),
\qquad
H(x)=-x\log x-(1-x)\log(1-x).
\]
\end{lemma}
\begin{proof}
This follows from Stirling's formula with explicit error bounds. One convenient form (valid for all $n\ge 1$) is
\[
\sqrt{2\pi}\,n^{n+\frac12}e^{-n}\ \le\ n!\ \le\
\sqrt{2\pi}\,n^{n+\frac12}e^{-n}\,e^{\frac1{12n}}.
\]
Applying these bounds to $q!,k!, (q-k)!$ yields
\[
\binom{q}{k}
\ \ge\
\frac{1}{\sqrt{2\pi}}\sqrt{\frac{q}{k(q-k)}}\,
\exp\!\Bigl(q\log q-k\log k-(q-k)\log(q-k)\Bigr)\,
\exp\!\Bigl(-\frac1{12k}-\frac1{12(q-k)}\Bigr).
\]
If $|k-xq|\le 2$, then $k=xq+O(1)$ and $q-k=(1-x)q+O(1)$, so the exponential term equals $\exp(qH(x))$
up to a multiplicative factor bounded away from $0$ for large $q$, while the prefactor is $\asymp 1/\sqrt{q}$
with constants depending on $x$. Absorbing these into $c_x>0$ yields the result.
\end{proof}

\section{Proof of the main theorem}
\begin{proof}[Proof of Theorem~\ref{thm:main}]
Fix $\alpha>0$ and $\varepsilon>0$.
%\smallskip\noindent
%\textbf{Choosing an almost-minimizer and error parameters.}
Choose $x\in(0,\tfrac12]$ such that
\[
\frac{2x}{H(x)^2}\ \le\ C_\star+\frac{\varepsilon}{4}.
\]
Choose $\eta>0$ so that
\begin{equation}\label{eq:eta-choice}
2x\Bigl(\frac{1}{H(x)}+\eta\Bigr)^2\ \le\ \frac{2x}{H(x)^2}+\frac{\varepsilon}{4}.
\end{equation}
By continuity at $\delta=0$, we may choose $\delta\in(0,1/2)$ sufficiently small so that
\begin{equation}\label{eq:delta-choice}
(1+\delta)(2+4\delta)(x+\delta)\Bigl(\frac{1}{H(x)}+\eta+\delta\Bigr)^2
\ \le\ 2x\Bigl(\frac{1}{H(x)}+\eta\Bigr)^2+\frac{\varepsilon}{4}.
\end{equation}
%\smallskip\noindent
%\textbf{Defining $(p,q,k)$ and embedding into $\Sigma_g$.}

For $g$ sufficiently large, define
\[
q\ :=\ \Bigl\lceil \Bigl(\frac{1}{H(x)}+\eta\Bigr)\alpha\log g\Bigr\rceil,
\qquad
k\ :=\ \text{an odd integer with }|k-xq|\le 2,
\qquad
p\ :=\ \Bigl\lfloor \frac{2g-2}{q-1}\Bigr\rfloor+1.
\]
Then $(p-1)(q-1)\le 2g-2$, hence $|\chi(\Sigma(p,q))|\le 2g-2$, and by Lemma~\ref{lem:embed} the surface $\Sigma(p,q)$
embeds as a subsurface of $\Sigma_g$. We regard $\Gamma(p,q;k)$ as a curve system on $\Sigma_g$.
%\smallskip\noindent
%\textbf{Producing more than $\lfloor g^{1+\alpha}\rfloor$ curves.}

By Lemma~\ref{lem:entropy} and the definition of $q$, for $g$ large,
\[
\binom{q}{k}\ \ge\ \frac{c_x}{\sqrt{q}}\,\exp\!\bigl(qH(x)\bigr)
\ \ge\ \frac{c_x}{\sqrt{q}}\,
\exp\!\Bigl(\alpha\log g+\eta\alpha H(x)\log g\Bigr)
\ =\ \frac{c_x}{\sqrt{q}}\,g^\alpha\,g^{\eta\alpha H(x)}.
\]
Since $q=O(\log g)$, we may assume $q\le \delta g$ by taking $g$ large. Then
$p-1=\lfloor\frac{2g-2}{q-1}\rfloor\ge \frac{2g-2}{q-1}-1 \ge \frac{g}{q}$ for large $g$, and hence
\[
M(p,q;k)\ =\ (p-1)\binom{q}{k}
\ \ge\ \frac{g}{q}\cdot \frac{c_x}{\sqrt{q}}\,g^\alpha\,g^{\eta\alpha H(x)}
\ =\ c_x\, g^{1+\alpha}\,\frac{g^{\eta\alpha H(x)}}{q^{3/2}}.
\]
As $q\asymp \log g$, the factor $g^{\eta\alpha H(x)}/q^{3/2}\to\infty$, so $M(p,q;k)>\lfloor g^{1+\alpha}\rfloor$
for all sufficiently large $g$.

%\smallskip\noindent
%\textbf{Extracting exactly $m=\lfloor g^{1+\alpha}\rfloor$ curves.}

Let $M:=M(p,q;k)$ and $m:=\lfloor g^{1+\alpha}\rfloor$. Choose a subset $\overline{\Gamma}\subset\Gamma(p,q;k)$ of size $m$
uniformly at random. Each pair of curves from $\Gamma(p,q;k)$ is selected with probability $\binom{m}{2}/\binom{M}{2}$, hence
\[
\mathbb{E}\bigl[\mathrm{cr}(\overline{\Gamma})\bigr]
\ =\ \frac{\binom{m}{2}}{\binom{M}{2}}\ \mathrm{cr}(\Gamma(p,q;k)).
\]
By the probabilistic method, there exists a specific choice of $\overline{\Gamma}$ with $\#\overline{\Gamma}=m$ such that
\begin{equation}\label{eq:subset}
\mathrm{cr}(\overline{\Gamma})
\ \le\ \frac{m(m-1)}{M(M-1)}\ \mathrm{cr}(\Gamma(p,q;k)).
\end{equation}
Using Lemma~\ref{lem:crossing} and $M\ge m+1$, we obtain
\[
\mathrm{cr}(\overline{\Gamma})
\ \le\ \frac{m(m-1)}{M(M-1)}\cdot \frac{4k}{p-1}M^2
\ =\ \frac{4k}{p-1}\,m(m-1)\,\frac{M}{M-1}.
\]
For $g$ large we have $m\ge 1/\delta$, hence $M/(M-1)\le 1+1/m\le 1+\delta$, and therefore
\begin{equation}\label{eq:cr-reduced}
\mathrm{Cr}(g,m)\ \le\ \mathrm{cr}(\overline{\Gamma})
\ \le\ (1+\delta)\,\frac{4k}{p-1}\,m^2
\ \le\ (1+\delta)\,\frac{4k}{p-1}\,g^{2+2\alpha}.
\end{equation}

%\smallskip\noindent
%\textbf{Estimating $\frac{4k}{p-1}$.}

From $p-1=\lfloor\frac{2g-2}{q-1}\rfloor$ we have
\[
p-1\ \ge\ \frac{2g-2}{q-1}-1\ =\ \frac{2g-q-1}{q-1},
\qquad\text{hence}\qquad
\frac{1}{p-1}\ \le\ \frac{q-1}{2g-q-1}\ \le\ \frac{q}{2g-q}.
\]
Assuming $q\le \delta g$, we have $q/(2g-q) \le (q/g)/(2-\delta)\le (\frac12+\delta)\,q/g$, and thus
\[
\frac{4k}{p-1}\ \le\ 4k\Bigl(\frac12+\delta\Bigr)\frac{q}{g}\ =\ (2+4\delta)\frac{kq}{g}.
\]
Since $|k-xq|\le 2$, for $g$ large we have $q\ge 2/\delta$ and hence $k\le xq+2\le (x+\delta)q$, so
\[
\frac{4k}{p-1}\ \le\ (2+4\delta)(x+\delta)\frac{q^2}{g}.
\]
Finally, from the definition of $q$ and taking $g$ large enough that $1\le \delta\,\alpha\log g$, we have
\[
q\ \le\ \Bigl(\frac{1}{H(x)}+\eta\Bigr)\alpha\log g +1
\ \le\ \Bigl(\frac{1}{H(x)}+\eta+\delta\Bigr)\alpha\log g.
\]
Hence
\[
\frac{4k}{p-1}
\ \le\ (2+4\delta)(x+\delta)\Bigl(\frac{1}{H(x)}+\eta+\delta\Bigr)^2\,
\frac{\alpha^2(\log g)^2}{g}.
\]
Combining this with \eqref{eq:cr-reduced} yields, for all sufficiently large $g$,
\[
\mathrm{Cr}\bigl(g,\lfloor g^{1+\alpha}\rfloor\bigr)
\ \le\
(1+\delta)(2+4\delta)(x+\delta)\Bigl(\frac{1}{H(x)}+\eta+\delta\Bigr)^2\,
\alpha^2\,g^{1+2\alpha}(\log g)^2.
\]
By \eqref{eq:delta-choice} and \eqref{eq:eta-choice}, the constant on the right-hand side is at most
\[
2x\Bigl(\frac{1}{H(x)}+\eta\Bigr)^2+\frac{\varepsilon}{4}
\ \le\ \frac{2x}{H(x)^2}+\frac{\varepsilon}{2}
\ \le\ C_\star+\varepsilon,
\]
which completes the proof.
\end{proof}

\section{The optimized constant and its interpretation}
This section explains the constant $C_\star$ from two complementary viewpoints: a scaling heuristic that leads to
$f(x)=2x/H(x)^2$, and a calculus condition that identifies critical points. It also clarifies in what sense $C_\star$
captures the best leading constant produced by this entropy-balance optimization of the Baader--J\"org--Parlier framework.

Fix $x\in(0,\tfrac12]$ and suppose one chooses parameters so that $k\approx xq$ and $q\approx (\alpha/H(x))\log g$.
Then $\binom{q}{k}$ grows like $\exp(qH(x))\approx g^\alpha$ (up to polynomial factors), while the embedding constraint
$(p-1)(q-1)\lesssim 2g$ suggests $p-1\approx 2g/q$. Lemma~\ref{lem:crossing} then gives a leading contribution of order
\[
\frac{4k}{p-1}\ \approx\ 2\,\frac{kq}{g}\ \approx\ 2x\,\frac{q^2}{g}
\ \approx\ \frac{2x}{H(x)^2}\,\frac{\alpha^2(\log g)^2}{g}.
\]
After extracting $m\approx g^{1+\alpha}$ curves, this leads to the heuristic leading constant
\[
f(x)\ :=\ \frac{2x}{H(x)^2}.
\]
Theorem~\ref{thm:main} makes this precise (up to an arbitrary $\varepsilon>0$), and optimizing over $x$ yields $C_\star$.

\begin{proposition}[Critical point equation]\label{prop:critical}
Let $f(x)=2x/H(x)^2$ on $(0,\tfrac12]$, where $H(x)=-x\log x-(1-x)\log(1-x)$.
If $x_0\in(0,\tfrac12)$ is a critical point of $f$, then it satisfies
\[
H(x_0)\ =\ 2x_0\log\Bigl(\frac{1-x_0}{x_0}\Bigr).
\]
In particular, using the expression for $H(x)$, this condition simplifies to
\[
x_0 \log x_0 = (1+x_0) \log(1-x_0).
\]
\end{proposition}
\begin{proof}
We have $H'(x)=\log\bigl(\frac{1-x}{x}\bigr)$. Differentiating $f(x)=2xH(x)^{-2}$ gives
\[
f'(x)=\frac{2}{H(x)^2}-\frac{4xH'(x)}{H(x)^3}
=\frac{2}{H(x)^3}\bigl(H(x)-2xH'(x)\bigr).
\]
Thus $f'(x_0)=0$ implies $H(x_0)=2x_0H'(x_0)=2x_0\log(\frac{1-x_0}{x_0})$.
Substituting $H(x_0) = -x_0\log x_0 - (1-x_0)\log(1-x_0)$, we get:
\[
-x_0\log x_0 - (1-x_0)\log(1-x_0) = 2x_0\log(1-x_0) - 2x_0\log x_0.
\]
Rearranging terms yields $x_0\log x_0 = (1+x_0)\log(1-x_0)$.
\end{proof}

\begin{remark}[Numerical minimizer]\label{rem:minimizer}
Since $H$ is continuous and positive on $(0,\tfrac12]$, the function $f$ is continuous there. Moreover, as $x\to 0^+$ one has
$H(x)\sim x\log(1/x)$, hence $f(x)\to\infty$, while $f(1/2)=1/(\log 2)^2\approx 2.08137$. Therefore $f$ attains its infimum
on $(0,\tfrac12]$. Any interior minimizer satisfies Proposition~\ref{prop:critical}. Numerically, solving the critical point equation
gives a minimizer near
\[
x_0\approx 0.2414851418,
\qquad\text{and}\qquad
C_\star=f(x_0)\approx 1.5805443269.
\]
\end{remark}

\begin{remark}[Optimality within this parameter-optimization scheme]\label{rem:optimality-scheme}
The proof of Theorem~\ref{thm:main} proceeds by fixing a ratio $x=k/q$, choosing $q$ on the scale
$q\sim (\alpha/H(x))\log g$ so that $\binom{q}{k}$ supplies the factor $g^\alpha$, and then taking $p$ near the embedding limit
$p-1\asymp g/q$. In that regime, the crossing estimate of Lemma~\ref{lem:crossing} yields an asymptotic constant arbitrarily close to
$f(x)=2x/H(x)^2$. Consequently, minimizing $f(x)$ over $x\in(0,\tfrac12]$ gives $C_\star$.
In this sense, $C_\star$ is the best leading constant obtainable by optimizing $k/q$ and the logarithmic scale of $q$ through this
entropy balance within the Baader--J\"org--Parlier construction \cite{BJP}. We do not claim that $C_\star$ is optimal among all possible
constructions of curve systems on $\Sigma_g$.
\end{remark}

\section*{Acknowledgements}
This work was supported by the National Research Foundation of Korea (NRF) grant funded by the Korean government (MSIT)
(No.\ RS-2025-00513595).

\end{document}